\documentclass[10pt]{article}
\usepackage{amsmath,amsthm,amsfonts,latexsym,amsopn,verbatim,amscd,amssymb}
\theoremstyle{plain}
\newtheorem{theorem}{Theorem}[section]
\newtheorem{lemma}[theorem]{Lemma}
\newtheorem{corollary}[theorem]{Corollary}

\newtheorem{proposition}[theorem]{Proposition}

\newcommand{\bnum}{\begin{enumerate}}
\newcommand{\enum}{\end{enumerate}}

\numberwithin{equation}{section}

\DeclareMathOperator{\Aut}{Aut}
\DeclareMathOperator{\Inn}{Inn}
\DeclareMathOperator{\orb}{orb}

\begin{document}

\title{\textbf{Autocommuting probability of a finite group}}
\author{Parama Dutta and Rajat Kanti Nath\footnote{Corresponding author}}
\date{}
\maketitle
\begin{center}\small{\it
Department of Mathematical Sciences, Tezpur University,\\ Napaam-784028, Sonitpur, Assam, India.\\

%Department of Mathematical Sciences,   Tezpur University,\\ Napaam-784028, Sonitpur, Assam, India.\\

%Department of Mathematical Sciences,   Tezpur University,\\ Napaam-784028, Sonitpur, Assam, India.\\

Emails:\, parama@gonitsora.com and rajatkantinath@yahoo.com}
\end{center}

\medskip

\begin{abstract}
Let $G$ be a finite group and $\Aut(G)$ the automorphism group of $G$. The autocommuting probability of $G$, denoted by $\Pr(G, \Aut(G))$, is the probability that  a randomly chosen automorphism of  $G$ fixes a  randomly chosen element of $G$. In this paper, we study $\Pr(G, \Aut(G))$ through a generalization. We obtain a computing formula, several bounds and characterizations of $G$ through $\Pr(G, \Aut(G))$. We conclude the paper by showing that the generalized autocommuting probability of $G$ remains unchanged under autoisoclinism.
\end{abstract}

\medskip

\noindent {\small{\textit{Key words:} Automorphism group, Autocommuting probability, Autoisoclinism.}


\medskip

\section{Introduction}

Let  $G$ be a finite group and $\Aut(G)$ be its automorphism group. For any  $x \in G$ and $\alpha \in \Aut(G)$ the element  $x^{-1}\alpha (x)$, denoted by $[x,\alpha]$,  is called an autocommutator in $G$. We write $S(G,\Aut(G))$ to denote the set  $\{[x,\alpha] : x \in G \text{ and } \alpha \in \Aut(G)\}$ and $K(G) := \langle S(G,\Aut(G)) \rangle$. We also write
$L(G) := \{x : [x, \alpha] = 1 \text{ for all }\alpha \in \Aut(G)\}$. Note that $K(G)$ and $L(G)$ are characteristic subgroups of $G$ known as the autocommutator subgroup and absolute center  of $G$. It is easy to see that $K(G)$ contains the commutator subgroup $G'$  and $L(G)$ is contained in the center $Z(G)$ of $G$. Further, $L(G) = \underset{\alpha \in \Aut(G)}{\cap}C_G(\alpha)$, where   $C_G(\alpha) = \{x\in G:[x,\alpha] = 1\}$ is a subgroup of $G$ known as acentralizer of $\alpha \in \Aut(G)$. The  subgroups $K(G)$ and $L(G)$ were defined and studied by
Hegarty in \cite{hegarty}. Let $C_{\Aut(G)}(x) := \{\alpha\in \Aut(G) : \alpha(x) = x\}$ for  $x \in G$ then $C_{\Aut(G)}(x)$ is a subgroup of $\Aut(G)$ and $C_{\Aut(G)}(G) = \{\alpha \in \Aut(G) : \alpha (x)  = x \text{ for all } x \in G\}$. It follows that $C_{\Aut(G)}(G) = \underset{x\in G}{\cap}C_{\Aut(G)}(x)$.

Let $G$ be a finite group acting on a set $\Omega$.  In the year 1975, Sherman \cite{sherman} introduced  the probability that  a randomly chosen element of $\Omega$  fixes a  randomly chosen element of $G$. We denote this probability by $\Pr(G, \Omega)$. If $\Omega = \Aut(G)$ then $\Pr(G, \Aut(G))$ is nothing but the  probability that the autocommutator of a randomly chosen pair of elements, one from $G$ and the other from $\Aut(G)$, is equal to $1$ (the identity element of $G$). Thus
\begin{equation}\label{AutComDeg}
{\Pr}(G,\Aut(G)) = \frac {\left|\{(x,\alpha)\in G\times \Aut(G):[x,\alpha]=1\}\right|}{|G||\Aut(G)|}.
\end{equation}
${\Pr}(G,\Aut(G))$ is called autocommuting probability of $G$. Sherman \cite{sherman} considered the case when $G$ is abelian and $\Omega = \Aut(G)$.
% Moghaddam et al. \cite{moga} and Rismanchian  \cite{Rismanchian15} also considered ${\Pr}(G,\Aut(G))$ and obtained few bounds for ${\Pr}(G,\Aut(G))$.
Note that if we take $\Omega = \Inn(G)$, the inner automorphim group of $G$, then ${\Pr}(G,\Inn(G))$ gives the probability that a randomly chosen pair of elements of $G$ commute.  ${\Pr}(G,\Inn(G))$ is known as commuting probability of $G$. It is also denoted by $\Pr(G)$. The study of $\Pr(G)$ was initiated by Erd$\ddot{\rm o}s$ and Tur$\acute {\rm a}$n \cite{pEpT68}. After Erd$\ddot{\rm o}s$ and Tur$\acute {\rm a}$n  many authors have worked on  $\Pr(G)$ and its generalizations (conf. \cite{Dnp13} and the references therein).
% Recently, Arora and Karan \cite{aK16} also considered this probability.
Somehow the study of ${\Pr}(G,\Aut(G))$ was neglected for long time. At this moment we have only a handful of papers on ${\Pr}(G,\Aut(G))$ (see  \cite{aK16}, \cite{moga}, \cite{Rismanchian15}). In this paper, we study ${\Pr}(G,\Aut(G))$ through a generalization.

In the year 2008, Pournaki and Sobhani \cite{PS08} have generalized $\Pr(G)$ and considered the probability that the commutator of a randomly chosen pair of elements of $G$ equals a given element $g \in G$. We write ${\Pr}_g(G)$ to denote this probability.  Motivated by their work, in this paper, we consider the probability  that the autocommutator of a randomly chosen pair of elements, one from $G$ and the other from $\Aut(G)$, is equal to a given element $g \in G$. We write ${\Pr}_g(G,\Aut(G))$ to denote this probability. Thus,
\begin{equation}\label{GenAutComDeg}
{\Pr}_g(G,\Aut(G)) = \frac {\left|\{(x, \alpha)\in G \times \Aut(G) : [x, \alpha] = g\}\right|}{|G||\Aut(G)|}.
\end{equation}
Notice that  ${\Pr}_1(G,\Aut(G)) = {\Pr}(G,\Aut(G))$. Hence ${\Pr}_g(G,\Aut(G))$ is a generalization of ${\Pr}(G,\Aut(G))$. Clearly, ${\Pr}_g(G, \Aut(G)) = 1$ if and only if $K(G) = \{1\}$ and $g = 1$  if and only if $G = L(G)$ and $g = 1$. Also, ${\Pr}_g(G, \Aut(G)) = 0$ if and only if  $g \notin S(G,\Aut(G))$. Therefore, we consider $G \ne L(G)$ and $g  \in S(G,\Aut(G))$ throughout the paper.

In Section 2,  we obtain a computing formula for ${\Pr}_g(G,\Aut(G))$ and deduce some of its consequences.  In Section 3, we obtain several bounds for ${\Pr}_g(G,\Aut(G))$. In Section 4, we obtain some characterizations of $G$ through ${\Pr}(G,\Aut(G))$. Finally, in the last section, we show that ${\Pr}_g(G,\Aut(G))$ is an invariant under autoisoclinism of groups.

\section{A computing formula}

%If $g=1$, then
%\begin{align*}
%{\Pr}(G,\Aut(G))=&\frac {\left|\{(x,\alpha)\in G\times \Aut(G):[x,\alpha]=1\}\right|}{|G||\Aut(G)|}\\
%=&\frac {\left|\{(x,\alpha)\in G\times \Aut(G):x^{-1}\alpha (x)=1\}\right|}{|G||\Aut(G)|}\\
%=&\frac {\left|\{(x,\alpha)\in G\times \Aut(G):\alpha (x)=x\}\right|}{|G||\Aut(G)|}.
%\end{align*}
%which is the probability defined by Sherman \cite{sherman} and Rismanchian et. al. \cite{Rismanchian15}.

For any $x, g \in G$ let  $T_{x, g}$ denotes the set $\{\alpha \in \Aut(G):[x,\alpha] = g\}$. Note that $T_{x, 1} = C_{\Aut(G)}(x)$. The following two lemmas play  a crucial role in obtaining the  computing formula for ${\Pr}_g(G,\Aut(G))$.
\begin{lemma}\label{lemma1}
Let $G$ be a finite group. If $T_{x, g} \ne \phi$ then $T_{x, g} = \sigma C_{\Aut(G)}(x)$ for some $\sigma\in T_{x, g}$. Hence, $|T_{x, g}| = |C_{\Aut(G)}(x)|$.
\end{lemma}
\begin{proof}
Let $\sigma\in T_{x, g}$ and $\beta \in \sigma C_{\Aut(G)}(x)$. Then $\beta = \sigma \alpha$ for some $\alpha\in C_{\Aut(G)}(x)$. We have
\[
[x,\beta] = [x,\sigma \alpha] = x^{-1}\sigma(\alpha (x)) = [x,\sigma] = g.
\]
Therefore, $\beta \in T_{x, g}$ and so $\sigma  C_{\Aut(G)}(x)  \subseteq T_{x, g}$. Again, let $\gamma\in T_{x, g}$ then $\gamma(x) = xg$. We have  $\sigma ^{-1} \gamma(x) = \sigma^{-1}(xg) = x$ and so $\sigma^{-1}  \gamma \in C_{\Aut(G)}(x)$. Therefore, $\gamma\in \sigma  C_{\Aut(G)}(x)$ which gives $T_{x, g} \subseteq \sigma C_{\Aut(G)}(x)$. Hence, the result follows.
\end{proof}

 We know that $\Aut(G)$ acts on $G$ by the action $(\alpha,x)\mapsto \alpha(x)$  where $\alpha \in \Aut(G)$ and $x\in G$.  Let $\orb(x) := \{\alpha(x) : \alpha \in \Aut(G)\}$ be the orbit of $x \in G$. Then by orbit-stabilizer theorem, we have
\begin{equation}\label{orbit-stabilizer-Thm}
 |\orb(x)| = \frac {|\Aut(G)|}{|C_{\Aut(G)}(x)|}.
\end{equation}
\begin{lemma}\label{lemma01}
Let $G$ be a finite group. Then $T_{x, g} \ne \phi$ if and only if  $xg \in  \orb(x)$.
\end{lemma}
\begin{proof}
The result follows from the fact that  $\alpha \in T_{x, g}$ if and only if $xg \in  \orb(x)$.
\end{proof}

Now we state and prove the main theorem of this section.
\begin{theorem}\label{thm2.4}
Let $G$ be a finite group. If  $g \in G$ then
\[
{\Pr}_g(G,\Aut(G)) = \frac {1}{|G||\Aut(G)|}\underset{xg\in \orb(x)}{\underset{x\in G}{\sum}}|C_{\Aut(G)}(x)| = \frac {1}{|G|}\underset{xg\in \orb(x)}{\underset{x\in G}{\sum}}\frac {1}{|\orb(x)|}.
\]
\end{theorem}
\begin{proof}
We have $\{(x,\alpha) \in G \times \Aut(G) : [x, \alpha] = g\} = \underset{x \in G}{\sqcup}(\{x\}\times T_{x, g})$, where $\sqcup$ represents the union of disjoint sets. Therefore, by \eqref{GenAutComDeg}, we have
\[
|G||\Aut(G)|{\Pr}_g(G,\Aut(G)) =  |\underset{x \in G}{\sqcup}(\{x\}\times T_{x, g})| = \underset{x \in G}{\sum}|T_{x, g}|.
\]
Hence, the result follows from Lemma \ref{lemma1}, Lemma \ref{lemma01} and \eqref{orbit-stabilizer-Thm}.
\end{proof}

Taking $g = 1$, in Theorem \ref{thm2.4}, we get the following corollary.
\begin{corollary}\label{formula2}
Let $G$ be a finite group. Then
\[
{\Pr}(G,\Aut(G)) = \frac {1}{|G||\Aut(G)|}\underset{x\in G}{\sum}|C_{\Aut(G)}(x)| = \frac {|\orb(G)|}{|G|}
\]
where $\orb(G) = \{\orb(x) : x \in G\}$.
\end{corollary}

\begin{corollary}
Let $G$ be a finite group. If $C_{\Aut (G)}(x)=\{I\}$   for all $x \in G\setminus \{1\}$, where $I$ is the identity element of $\Aut(G)$, then
\[
\Pr(G,\Aut(G)) = \frac {1}{|G|} + \frac {1}{|\Aut(G)|} - \frac {1}{|G||\Aut(G)|}.
\]
\end{corollary}
\begin{proof}
By Corollary \ref{formula2}, we have
\[
|G||\Aut(G)|\Pr(G,\Aut(G))  = \underset{x\in G}{\sum}|C_{\Aut(G)}(x)|
  =   |\Aut(G)| + |G| - 1.
\]
Hence, the result follows.
\end{proof}
\noindent
We also have
$|\{(x,\alpha) \in G \times \Aut(G) : [x, \alpha] = 1\}| = \underset{\alpha \in \Aut(G)}{\sum}|C_G(\alpha)|$ and hence
\begin{equation}
{\Pr}(G,\Aut(G)) = \frac {1}{|G||\Aut(G)|}\underset{\alpha \in \Aut(G)}{\sum}|C_G(\alpha)|.
\end{equation}

We conclude this section with the following two results.
\begin{proposition}
Let $G$ be a finite group. If $g \in G$ then
\[
{\Pr}_{g^{-1}}(G, \Aut(G)) = {\Pr}_g(G, \Aut(G)).
\]
\end{proposition}

\begin{proof}
  Let 
  \begin{align*}
   A &= \{(x, \alpha) \in G \times \Aut(G) : [x, \alpha] = g\} \text{ and }\\
   B &= \{ (y, \beta) \in G \times \Aut(G) : [y, \beta] = g^{-1} \}.
  \end{align*}
  Then $(x, \alpha) \mapsto ( \alpha(x), \alpha ^{-1})$ gives a bijection between $A$ and $B$. Therefore $|A| = |B|$. Hence the result follows from \eqref{GenAutComDeg}.
  
\end{proof}

%
%\begin{proof}
%By \eqref{GenAutComDeg}, We have
%\begin{align*}
%|G||\Aut(G)|{\Pr}_g(G &\times \Aut(G))  = \left|\{(x, \alpha)\in G \times \Aut(G) : [x,\alpha] = g\}\right|\\
%& = \left|\{(x, \alpha)\in G \times \Aut(G) : [x,\alpha]^{-1} = g^{-1}\}\right|\\
%& = \left|\{(\alpha(x), \alpha^{-1}) \in G \times \Aut(G) : [\alpha(x),\alpha^{-1}] = g^{-1}\}\right|\\
%& = \left|\{(y, \beta) \in G \times \Aut(G) : [y, \beta] = g^{-1}\}\right|\\
%& = |G||\Aut(G)|{\Pr}_{g^{-1}}(G \times \Aut(G)).
%\end{align*}
%Hence, the result follows.
%%For all $x \in G$, $\alpha \in \Aut(G)$ and $g\in G$; $[x,\alpha]=g$ implies $[\alpha (x),\alpha ^{-1}]=g^{-1}$ where $\alpha (x)\in \alpha (G)$. Clearly, $\alpha (G)\cong G$ and hence
%%\begin{align*}
%%{\Pr}_g(G\times \Aut(G))&=\frac {\left|\{(x,\alpha)\in G\times \Aut(G):[x,\alpha]=g\}\right|}{|G||\Aut(G)|}\\
%%&=\frac {\left|\{(x,\alpha)\in G\times \Aut(G):[x,\alpha]^{-1}=g^{-1}\}\right|}{|G||\Aut(G)|}\\
%%&=\frac {\left|\{(\alpha (x),\alpha^{-1})\in \alpha (G)\times \Aut(G):[\alpha (x),\alpha^{-1}]=g^{-1}\}\right|}{|\alpha (G)||\Aut(G)|}\\
%%&=\frac {\left|\{(y,\beta)\in G\times \Aut(G):[y,\beta]=g^{-1}\}\right|}{|G||\Aut(G)|}\\
%%&={\Pr}_{g^{-1}}(G\times \Aut(G))
%%\end{align*}
%\end{proof}

\begin{proposition}
Let $G$ and $H$ be two finite groups such that $\gcd(|G|,|H|) = 1$.  If $(g, h) \in G \times H$ then
\[
{\Pr}_{(g, h)}(G \times H,\Aut(G \times H)) = {\Pr}_{g}(G, \Aut(G)) {\Pr}_{h}(H, \Aut(H)).
\]
\end{proposition}
\begin{proof}
Let
\begin{align*}
\mathcal{X} &= \{((x, y), \alpha_{G \times H}) \in (G \times H) \times \Aut(G \times H) : [(x, y), \alpha_{G \times H}] = (g, h)\},\\
\mathcal{Y} &= \{(x, \alpha_G) \in G \times \Aut(G) : [x, \alpha_G] = g\} \text{ and }\\
\mathcal{Z} &= \{(y, \alpha_H) \in H \times \Aut(H) : [y, \alpha_H] = h\}.
\end{align*}
Since  $\gcd(|G|,|H|) = 1$, by \cite[Lemma 2.1]{cD07},  we have $\Aut(G \times H) = \Aut(G)\times \Aut(H)$. Therefore, for every $\alpha_{G \times H} \in \Aut(G \times H)$ there exist unique $\alpha_G \in  \Aut(G)$  and $\alpha_H \in  \Aut(H)$ such that $\alpha_{G \times H} = \alpha_G \times \alpha_H$, where $\alpha_G \times \alpha_H((x, y)) = (\alpha_G(x), \alpha_H(y))$ for all $(x, y) \in G \times H$. Also, for all $(x, y) \in G \times H$, we have $[(x, y), \alpha_{G \times H}] = (g, h)$ if and only if $[x, \alpha_G] = g$ and $[y, \alpha_H] = h$. These leads to show that $\mathcal{X} = \mathcal{Y} \times \mathcal{Z}$. Therefore
\[
\frac{|\mathcal{X}|}{|G \times H||\Aut(G \times H)|}
= \frac{|\mathcal{Y}|}{|G||\Aut(G)|}\cdot\frac{|\mathcal{Z}|}{|H||\Aut(H)|}.
\]
Hence, the result follows from \eqref{AutComDeg}.
%From Lemma 2.1 of \cite{moga}, if $gcd(|G|,|G_2|)=1$ then $\Aut(G\times G_2)=\Aut(G)\times \Aut(G_2)$ and hence the second part follows.
\end{proof}
%This is a generalization of the Theorem 2.1 of \cite{moga}.

\section{Some bounds}
 We begin with the following lower bounds.

\begin{proposition}
Let $G$ be a finite group. Then
\begin{enumerate}
\item  
${\Pr}_g(G,\Aut(G))\geq \frac {|L(G)|}{|G|} +  \frac {|C_{\Aut(G)}(G)|(|G|-|L(G)|)}{|G||\Aut(G)|}$ if $g = 1$.
\item   
${\Pr}_g(G,\Aut(G))\geq \frac {|L(G)||C_{\Aut(G)}(G)|}{|G||\Aut(G)|}$ if  $g\neq 1$.
\end{enumerate}
\end{proposition}

\begin{proof}
Let $\mathcal{C} := \{(x,\alpha) \in G\times \Aut(G): [x,\alpha] = g\}$.

(a) We have $(L(G) \times \Aut(G))\cup (G \times C_{\Aut(G)}(G)) \subseteq \mathcal{C}$ and
$|(L(G) \times \Aut(G))\cup (G \times C_{\Aut(G)}(G))|$ is equal to $|L(G)||\Aut(G)| + |C_{\Aut(G)}(G)||G| - |L(G)||C_{\Aut(G)}(G)|$. Hence, the result follows from \eqref{GenAutComDeg}.

(b) Since $g\in S(G,\Aut (G))$ we have $\mathcal{C}$ is non-empty. Let $(y, \beta) \in \mathcal{C}$ then $(y, \beta) \notin  L(G)\times C_{\Aut(G)}(G)$ otherwise $[y, \beta] = 1$. It is easy to see that the coset $(y, \beta)(L(G)\times C_{\Aut(G)}(G))$ having order $|L(G)||C_{\Aut(G)}(G)|$ is a subset of $\mathcal{C}$. Hence, the result follows from \eqref{GenAutComDeg}.
 % ${\Pr}_g(G,\Aut(G))\neq 0$, there exist $(x,\alpha)\in G\times \Aut(G)$ such that $[x,\alpha]=g$. Again $g\neq 1$ implies $(x,\alpha) \notin L(G)\times C_{\Aut(G)}(G)$.
%Let us consider, $$T_{(x,\alpha)}=(L(G)\times C_{\Aut(G)}G)(x,\alpha).$$
%Clearly, $T_{(x,\alpha)}\subseteq \{(x,\alpha)\in G\times \Aut(G):[x,\alpha]=g \}$ and hence
%\begin{align*}
%{\Pr}_g(G,\Aut(G))&=\frac {\left|\{(x,\alpha)\in G\times \Aut(G):[x,\alpha]=g\}\right|}{|G||\Aut(G)|}\\
%&\geq \frac {|T_{(x,\alpha)}|}{|G||\Aut(G)|}\\
%&=\frac {|L(G))||C_{\Aut(G)}G|}{|G||\Aut(G)|}
%\end{align*}
\end{proof}

\begin{proposition}\label{prop3.2}
Let $G$ be a finite group. Then
\[
{\Pr}_g(G,\Aut(G)) \leq \Pr(G,\Aut(G)).
\]
The equality holds if and only if $g = 1$.
\end{proposition}
\begin{proof}
By Theorem \ref{thm2.4}, we have
\begin{align*}
{\Pr}_g(G, \Aut(G)) &= \frac {1}{|G||\Aut(G)|}\underset{xg\in \orb(x)}{\underset{x \in G}{\sum}}|C_{\Aut(G)}(x)|\\
&\leq \frac {1}{|G||\Aut(G)|}\underset{x \in G}{\sum}|C_{\Aut(G)}(x)| = \Pr(G,\Aut(G)).
\end{align*}
The equality holds if and only if $xg\in \orb(x)$ for all $x \in G$ if and only if $g = 1$.
\end{proof}

\begin{proposition}
Let $G$ be a finite group and $p$ the smallest prime dividing $|\Aut(G)|$. If  $g \neq 1$ then
\[
{\Pr}_g(G, \Aut(G))\leq \frac {|G| - |L(G)|}{p|G|} < \frac {1}{p}.
\]
\end{proposition}
\begin{proof}
By Theorem \ref{thm2.4}, we have
\begin{equation}\label{eq3.1}
{\Pr}_g(G,\Aut(G)) = \frac {1}{|G|}\underset{xg\in \orb(x)}{\underset{x\in G \setminus L(G)}{\sum}}\frac {1}{|\orb(x)|}
\end{equation}
noting that for $x \in L(G)$ we have $xg \notin \orb(x)$. Also, for $x\in G \setminus L(G)$ and $xg \in \orb(x)$ we have $|\orb(x)| > 1$. Since $|\orb(x)|$ is a divisor of $|\Aut(G)|$ we have $|\orb(x)| \geq p$. Hence, the result follows from \eqref{eq3.1}.
% Hence we assume $L(G, \Aut(G))\neq G$. Suppose $x\in G$ such that $xg\in O_G(x)$. Since $g\neq 1$, we have $x\notin L(G))$. As $|O_G(x)|>1$ and $|O_G(x)|$ is a divisor of $|\Aut(G)|$ therefore $|O_G(x)|\geq p$. Therefore,
%\begin{align*}
%{\Pr}_g(G,\Aut(G))&=\frac {1}{|G|}\underset{xg\in O_G(x)}%{\underset{x\in G}{\sum}}\frac {1}{|O_G(x)|}\\
%&\leq \frac {1}{|G|}\underset{xg\in O_G(x)}{\underset{x\in G}{\sum}}\frac {1}{p}\\
%&= \frac {1}{|G|}\underset{xg\in O_G(x)}{\underset{x\in G\setminus L(G))}{\sum}}\frac {1}{p}\\
%&\leq \frac {1}{|G|}\underset{x\in G\setminus L(G))}{\sum}\frac {1}{p}\\
%&= \frac {|G|-|L(G))|}{p|G|}\\
%& < \frac {1}{p}.
%\end{align*}
\end{proof}

%\begin{proposition}
%If $G_1\subseteq G_2$ then
%$$
%{\Pr}_g(G_1,\Aut(G_2))\leq |G_2:G_1|{\Pr}_g(G_2,\Aut(G_2)).
%$$
%\end{proposition}
%\begin{proof}
%\begin{align*}
%|G_1||\Aut(G_2)|{\Pr}_g(G_1,\Aut(G_2))&=\underset{x\in G_1}{\sum}|C_{\Aut(G_2)}(x)|\\
%&\leq \underset{x\in G_2}{\sum}|C_{\Aut(G_2)}(x)|\\
%&=|G_2||\Aut(G_2)|{\Pr}_g(G_2,\Aut(G_2))~\textup{and hence}
%\end{align*}
%$${\Pr}_g(G_1,\Aut(G_2))\leq |G_2:G_1|{\Pr}_g(G_2,\Aut(G_2)).$$
%\end{proof}

The remaining part of this section is devoted in obtaining some lower and upper bounds for  $\Pr(G,\Aut(G))$. The following theorem is an improvement of \cite[Theorem 2.3 (ii)]{moga}.
\begin{theorem}\label{thm3.6}
Let $G$ be a finite group and $p$ the smallest prime dividing $|\Aut(G)|$.  Then
\[
\Pr(G,\Aut(G)) \geq \frac {|L(G)|}{|G|} + \frac {p(|G| - |X_G| - |L(G)|) + |X_G|}{|G||\Aut(G)|}
\]
and
\[
\Pr(G,\Aut(G)) \leq \frac {(p - 1)|L(G)| + |G|}{p|G|} - \frac {|X_G|(|\Aut(G)| - p)}{p|G||\Aut(G)|},
\]
where $X_G = \{x \in G : C_{\Aut(G)}(x) = \{I\}\}$.
\end{theorem}
\begin{proof}
We have   $X_G \cap L(G) = \phi$. Therefore
\[
\underset{x\in G}{\sum}|C_{\Aut(G)}(x)| = |X_G| + |\Aut(G)||L(G)| + \underset{x\in G\setminus (X_G\cup L(G))}{\sum}|C_{\Aut(G)}(x)|.
\]
%\begin{align*}
%\underset{x\in G}{\sum}|C_{\Aut(G)}(x)| & =
%&\underset{x\in X_G}{\sum}|C_{\Aut(G)}(x)|+\underset{x\in %L(G)}{\sum}|C_{\Aut(G)}(x)|\\
%&+\underset{x\in G\setminus (X_G\cup L(G))}{\sum}|%C_{\Aut(G)}(x)|\\
%&=\underset{x\in X_G}{\sum}1+\underset{x\in L(G)}{\sum}|\Aut(G)|\\
%&+\underset{x\in G\setminus (X_G\cup L(G))}{\sum}|C_{\Aut(G)}(x)|.
%\end{align*}
For $x  \in G \setminus (X_G \cup L(G))$ we have $\{I\}\neq C_{\Aut(G)}(x)\neq \Aut(G)$ which implies $p \leq |C_{\Aut(G)}(x)|\leq \frac {|\Aut(G)|}{p}$. Therefore
\begin{equation}\label{our_bound1}
\underset{x\in G}{\sum}|C_{\Aut(G)}(x)| \geq |X_G| + |\Aut(G)||L(G)| + p(|G| - |X_G| - |L(G)|)
\end{equation}
and
\begin{equation}\label{our_bound2}
\underset{x\in G}{\sum}|C_{\Aut(G)}(x)| \leq |X_G| + |\Aut(G)||L(G)| + \frac{|\Aut(G)|(|G| - |X_G| - |L(G)|)}{p}.
\end{equation}
Hence, the result follows from Corollary \ref{formula2}, \eqref{our_bound1} and \eqref{our_bound2}.
%\begin{align*}
%\Pr(G,\Aut(G))=&\frac {1}{|G||\Aut(G)|}(|X_G|+|\Aut(G)||L(G)|)\\
%&+\frac {1}{|G||\Aut(G)|}\underset{x\in G\setminus (X_G\cup L(G))}{\sum}|C_{\Aut(G)}(x)|\\
%\leq &\frac {|X_G|}{|G||\Aut(G)|}+\frac {|L(G)|}{|G|}+\frac {1}{|G||\Aut(G)|}\underset{x\in G\setminus (X_G\cup L(G))}{\sum}\frac {|\Aut(G)|}{p}\\
%=&\frac {|X_G|}{|G||\Aut(G)|}+\frac {|L(G)|}{|G|}+\frac {1}{p|G|}(|G|-|X_G\cup L(G)|)\\
%=&\frac {|X_G|}{|G||\Aut(G)|}+\frac {|L(G)|}{|G|}+\frac {1}{p|G|}(|G|-|X_G|-|L(G)|)\\
%=&\frac {(p-1)|L(G)|+|G|}{p|G|}-\frac {|X_G|(|\Aut(G)|-p)}{p|G||\Aut(G)|}.
%\end{align*}
%Similarly taking $|C_{\Aut(G)}(x)|\geq p$, the second inequality follows.
\end{proof}
%Note that this inequalities is better than the inequalities obtained in Theorem 2.3 of \cite{moga}.

We have the following two corollaries.
\begin{corollary}\label{bound_like3/4}
Let $G$ be a finite group. If $p$ and $q$ are the smallest primes dividing  $|\Aut(G)|$ and $|G|$ respectively then
\[
\Pr(G,\Aut(G)) \leq \frac{p + q - 1}{pq}.
\]
In particular, if $p = q$ then $\Pr(G,\Aut(G)) \leq \frac{2p - 1}{p^2} \leq \frac{3}{4}$.
\end{corollary}
\begin{proof}
Since $G \ne L(G)$ we have $|G : L(G)| \geq q$. Therefore, by Theorem \ref{thm3.6}, we have
\[
\Pr(G,\Aut(G)) \leq \frac{1}{p}\left(\frac{p - 1}{|G : L(G)|} + 1\right) \leq \frac{p + q - 1}{pq}.
\]
\end{proof}

\begin{corollary}\label{bound_like5/8}
Let $G$ be a finite group and $p$, $q$ be the smallest primes dividing  $|\Aut(G)|$ and $|G|$ respectively. If $G$ is non-abelian then
\[
\Pr(G,\Aut(G)) \leq \frac{q^2 + p - 1}{pq^2}.
\]
In particular, if $p = q$ then $\Pr(G,\Aut(G)) \leq \frac{p^2 + p - 1}{p^3} \leq \frac{5}{8}$.
\end{corollary}
\begin{proof}
Since $G$ is non-abelian we have $|G : L(G)| \geq q^2$. Therefore, by Theorem \ref{thm3.6}, we have
\[
\Pr(G,\Aut(G)) \leq \frac{1}{p}\left(\frac{p - 1}{|G : L(G)|} + 1\right) \leq \frac{q^2 + p - 1}{pq^2}.
\]
\end{proof}

\begin{theorem}\label{prop3.8}
Let $G$ be a finite group. Then
\[
\Pr(G,\Aut(G)) \geq \frac {1}{|S(G,\Aut(G))|}\left(1 + \frac {|S(G,\Aut(G))| - 1}{|G : L(G)|}\right).
\]
The equality holds if and only if $\orb(x) = xS(G,\Aut(G))$
 for all $x \in G \setminus L(G)$.
\end{theorem}
\begin{proof}
For all $x \in G \setminus L(G)$ we have $\alpha (x) = x[x, \alpha] \in xS(G,\Aut(G))$. Therefore $\orb(x) \subseteq xS(G,\Aut(G))$ and so $|\orb(x)| \leq |S(G,\Aut(G))|$
 for all $x \in G \setminus L(G)$. Now, by Corollary \ref{formula2}, we have
\begin{align*}
\Pr(G,\Aut(G))& = \frac {1}{|G|}\left(\underset{x \in L(G)}{\sum}\frac {1}{|\orb(x)|} + \underset{x \in G \setminus L(G)}{\sum}\frac {1}{|\orb(x)|}\right)\\
&\geq \frac {|L(G)|}{|G|} + \frac {1}{|G|}\underset{x\in G\setminus L(G)}{\sum}\frac{1}{|S(G,\Aut(G))|}.
%&=\frac {1}{|S(G,\Aut(G))|}\left(1+\frac {|S(G,\Aut(G))|-1}%{|G:L(G)|}\right)
\end{align*}
Hence, the result follows.
\end{proof}
The following lemma is useful in obtaining the next corollary.
\begin{lemma}\label{lemma4.4}
Let $G$ be a finite group. Then, for any two integers  $m \geq n$, we have
\[
\frac {1}{n}\left(1 + \frac{n - 1}{|G : L(G)|}\right) \geq \frac {1}{m}\left(1 + \frac{m - 1}{|G : L(G)|}\right).
\]
 If $L(G)\neq G$ then equality holds if and only if $m=n$.
\end{lemma}
\begin{proof}
The proof is an easy exercise.
\end{proof}

\begin{corollary}\label{lastcor}
Let $G$ be a finite group. Then
\[
{\Pr}(G,\Aut(G))\geq \frac {1}{|K(G)|}\left(1 + \frac {|K(G)| - 1}{|G : L(G)|}\right).
\]
If $G \ne L(G)$ then the equality holds if and only if $K(G) =  S(G,\Aut(G))$ and $\orb(x)$ is equal to $S(G,\Aut(G))$
 for all $x \in G \setminus L(G)$.
\end{corollary}
\begin{proof}
Since $|K(G)|  \geq |S(G, \,\Aut(G))|$, the result follows from Theorem \ref{prop3.8} and Lemma \ref{lemma4.4}.

Note that the equality holds if and only if equality holds
in Theorem \ref{prop3.8} and Lemma \ref{lemma4.4}.
\end{proof}
\noindent This lower bound is also obtained in \cite{aK16}. Note that the  lower bound   in Theorem \ref{prop3.8} is better than that in Corollary \ref{lastcor}. Also
\[
 \frac {1}{|K(G)|}\left(1 + \frac {|K(G)| - 1}{|G : L(G)|}\right)
\geq \frac
{|L(G)|}{|G|} + \frac {p(|G| - |L(G)|)}{|G||\Aut(G)|}.
\]
Hence, the  lower bound   in Theorem \ref{prop3.8} is also better than that in   \cite[Theorem 2.3(ii)]{moga}.
%Also, the  lower bound   in Theorem \ref{prop3.8} is better than that in Corollary \ref{lastcor}.

\section{Some Characterizations}
In \cite{aK16}, Arora and Karan obtain a characterization of a finite group $G$ if equality holds in Corollary \ref{lastcor}. In this section, we obtain some characterizations of $G$ if equality holds in Corollary \ref{bound_like3/4} and Corollary \ref{bound_like5/8}. We begin with the following result.
\begin{proposition}\label{prop4.1}
Let $G$ be a finite group with  $\Pr(G, \Aut(G)) = \frac {p + q -1}{pq}$ for some primes $p$ and $q$. Then  $pq$  divides $|G||\Aut(G)|$. If $p$ and $q$ are the smallest primes dividing $|\Aut(G)|$ and $|G|$ respectively, then
\[
\frac{G}{L(G)} \cong \mathbb Z_q.
\]
In particular, if $G$ and $\Aut (G)$ are of even order and $\Pr(G,\Aut(G)) = \frac{3}{4}$ then $\frac{G}{L(G)}\cong\mathbb Z_2$.
\end{proposition}
\begin{proof}
By \eqref{AutComDeg}, we have $(p + q -1)|G||\Aut(G)| = pq|\{(x,\alpha)\in G\times \Aut(G):[x,\alpha]=1\}|$. Therefore, $pq$ divides $|G||\Aut(G)|$.

If $p$ and $q$ are the smallest primes dividing $|\Aut(G)|$ and $|G|$ respectively then, by Theorem \ref{thm3.6}, we have
\[
\frac{p + q -1}{pq} \leq \frac{1}{p}\left(\frac{p - 1}{|G : L(G)|} + 1\right)
\]
which gives $|G : L(G)| \leq q$. Hence, $\frac{G}{L(G)} \cong \mathbb Z_q$.
\end{proof}

\begin{proposition}\label{prop4.2}
Let $G$ be a finite non-abelian group  with $\Pr(G,\Aut(G)) = \frac {q^2 + p - 1}{pq^2}$ for some primes $p$ and $q$. Then $pq$ divides  $|G||\Aut(G)|$. If $p$ and $q$ are the smallest primes dividing $|\Aut(G)|$ and $|G|$ respectively then
\[
\frac{G}{L(G)}\cong\mathbb Z_q \times \mathbb Z_q.
\]
In particular, if $G$ and $\Aut (G)$ are of even order and $\Pr(G,\Aut(G)) = \frac{5}{8}$ then $\frac{G}{L(G)} \cong \mathbb Z_2 \times \mathbb Z_2$.
\end{proposition}

\begin{proof}
By \eqref{AutComDeg}, we have $(q^2 + p -1)|G||\Aut(G)| = pq^2|\{(x,\alpha)\in G\times \Aut(G):[x,\alpha]=1\}|$. Therefore, $pq$ divides $|G||\Aut(G)|$.

If $p$ and $q$ are the smallest primes dividing $|\Aut(G)|$ and $|G|$ respectively then, by Theorem \ref{thm3.6}, we have
\[
\frac{q^2 + p -1}{pq^2} \leq \frac{1}{p}\left(\frac{p - 1}{|G : L(G)|} + 1\right)
\]
which gives $|G : L(G)| \leq q^2$. Since $G$ is non-abelian we have $|G : L(G)| \neq 1, q$. Hence, $\frac{G}{L(G)} \cong \mathbb Z_q \times \mathbb Z_q$.
\end{proof}

%The converse of the Proposition \ref{prop4.1} and \ref{prop4.2} are not true in general. However the from the following Proposition converse can be obtained.

We conclude this section with the following partial converses of Proposition \ref{prop4.1} and \ref{prop4.2}.

\begin{proposition}
Let $G$ be a finite group. Let $p, q$ be the smallest prime divisors of $|\Aut(G)|$ and $|G|$ respectively and $|\Aut(G) : C_{\Aut(G)}(x)| = p$ for all $x \in G \setminus L(G)$.
\begin{enumerate}
\item
If $\frac{G}{L(G)} \cong \mathbb Z_q$ then $\Pr(G,\Aut(G)) = \frac {p + q - 1}{pq}$.
\item
If $\frac{G}{L(G)} \cong\mathbb Z_q \times \mathbb Z_q$  then $\Pr(G,\Aut(G)) = \frac {q^2 + p - 1}{pq^2}$.
\end{enumerate}
\end{proposition}
\begin{proof}
Since $|\Aut(G) : C_{\Aut(G)}(x)| = p$ for all $x \in G \setminus L(G)$ we have $|C_{\Aut(G)}(x)| = \frac{|\Aut(G)|}{p}$ for all $x \in G \setminus L(G)$. Therefore, by Corollary \ref{formula2}, we have
\begin{align*}
\Pr(G,\Aut(G)) &= \frac{|L(G)|}{|G|} + \frac {1}{|G||\Aut(G)|} \underset{x \in G \setminus L(G)}{\sum}|C_{\Aut(G)}(x)|\\
&= \frac {|L(G)|}{|G|} + \frac{|G| - |L(G)|}{p|G|}.
\end{align*}
Thus
\begin{equation}\label{conv_eq}
\Pr(G,\Aut(G)) = \frac{1}{p}\left(\frac{p - 1}{|G : L(G)|}  + 1\right).
\end{equation}
(a) If $\frac{G}{L(G)} \cong \mathbb Z_q$ then \eqref{conv_eq} gives $\Pr(G,\Aut(G)) = \frac {p + q - 1}{pq}$.

\noindent (b) If $\frac{G}{L(G)} \cong\mathbb Z_q \times \mathbb Z_q$  then \eqref{conv_eq} gives $\Pr(G,\Aut(G)) = \frac {q^2 + p - 1}{pq^2}$.
\end{proof}

\section{Autoisoclinism of groups}
In the year 1940, Hall defined   isoclinism between two groups (see \cite{pH40}). In the year
1995, Lescot \cite{pL95} showed that the commuting probability of two isoclinic finite groups are same. Later on  Pournaki and Sobhani \cite{PS08} showed that ${\Pr}_g(G) = {\Pr}_{\beta(g)}(H)$ if $(\alpha, \beta)$ is an isoclinism between the finite groups $G$ and $H$. Following Hall, Moghaddam et al. \cite{msE13} have defined autoisoclinism between two groups. Recall that two groups $G$ and $H$ are said to be autoisoclinic  if there exist isomorphisms $\psi : \frac{G}{L(G)} \to \frac{H}{L(H)}$, $\beta : K(G)\to K(H)$ and $\gamma : \Aut(G) \to \Aut(H)$ such that the following diagram commutes
%\vspace{.25cm}
\begin{center}
$
\begin{CD}
   \frac{G}{L(G)} \times \Aut(G) @>\psi \times \gamma>> \frac{H}{L(H)} \times
   \Aut(H)\\
   @VV{a_{(G, \Aut(G))}}V  @VV{a_{(H, \Aut(H))}}V\\
   K(G) @> \beta >> K(H)
\end{CD}
$
\end{center}
where the maps $a_{(G, \Aut(G))}: \frac{G}{L(G)} \times \Aut(G) \to K(G)$ and $a_{(H, \Aut(H))}: \frac {H}{L(H)} \times \Aut(H) \to K(H)$  are given by
\[
a_{(G, \Aut(G))}(xL(G), \alpha_1) = [x, \alpha_1]
\text{ and }
a_{(H, \Aut(H))}(yL(H), \alpha_2) = [y, \alpha_2]
\]
respectively. In this case, the pair $(\psi\times\gamma, \beta)$ is called an autoisoclinism between the groups $G$
and $H$.

Recently, in \cite{Rismanchian15}, Rismanchain and Sepehrizadeh have shown that ${\Pr}(G,\Aut(G)) = {\Pr}(H,\Aut(H))$ if $G$ and $H$ are autoisoclinic finite groups. We conclude this paper with the following   generalization of  \cite[Lemma 2.5]{Rismanchian15}.
\begin{theorem}
Let $G$ and $H$ be two finite groups and  $(\psi \times \gamma, \beta)$  an autoisoclinism between them.  Then
\[
{\Pr}_g(G,\Aut(G)) = {\Pr}_{\beta (g)}(H,\Aut(H)).
\]
\end{theorem}
\begin{proof}
Let ${\mathcal{S}}_g = \{(xL(G),\alpha_1) \in \frac {G}{L(G)} \times \Aut(G) : [x,\alpha_1] = g\}$ and ${\mathcal{T}}_{\beta (g)} = \{(yL(H),\alpha_2)$ $\in \frac
{H}{L(H)} \times \Aut(H) : [y,\alpha_2] = \beta (g)\}$. Since $(\psi \times \gamma, \beta)$  an autoisoclinism between $G$ and $H$, the mapping $\lambda : {\mathcal{S}}_g \to {\mathcal{T}}_{\beta (g)}$ given by $(xL(G), \alpha_1) \mapsto (\psi(xL(G)), \gamma(\alpha_1))$ gives an one to one correspondence between ${\mathcal{S}}_g$ and ${\mathcal{T}}_{\beta (g)}$. Hence, $|\mathcal{S}_g|=|{\mathcal{T}}_{\beta (g)}|$. Also

\begin{equation}\label{autoclin_eq1}
|\{(x,\alpha_1) \in G \times \Aut(G) : [x, \alpha_1] = g\}| = |L(G)||\mathcal{S}_g|
\end{equation}
and
\begin{equation}\label{autoclin_eq2}
|\{(y, \alpha_2) \in H\times \Aut(H) : [y,\alpha_2] = \beta(g)\}| = |L(H)||{\mathcal{T}}_{\beta (g)}|.
\end{equation}
 Therefore, by  \eqref{AutComDeg} and \eqref{autoclin_eq1}, we have
\[
{\Pr}_g(G,\Aut(G)) = \frac {|L(G)||\mathcal{S}_g|}{|G||\Aut(G)|}.
\]
Since $|G/L(G)| = |H/L(H)|$ and $|\Aut(G)| = |\Aut(H)|$ we have
\[
{\Pr}_g(G,\Aut(G)) = \frac {|L(H)||{\mathcal{T}}_{\beta (g)}|}{|H||\Aut(H)|}.
\]
Hence, the result follows from \eqref{autoclin_eq2} and \eqref{AutComDeg}.
%\begin{align*}
%{\Pr}_g(G,\Aut(G))&=\frac {|\{(x,\alpha_1)\in G \times \Aut(G): [x,\alpha_1]=g\}|}{|G||\Aut(G)|}\\
%&=\frac {|L(G)||S_g|}{|G||\Aut(G)|}\\
%&=\frac {|L(H)||T_\beta (g)|}{|H||\Aut(K_2)
%|}\\
%&=\frac {|\{(y,\alpha_2)\in H\times \Aut(H): [y,\alpha_2]=\beta_(g)\}|}{|H||\Aut(H)|}\\
%&= {\Pr}_{\beta (g)}(H, \Aut(H)).
%\end{align*}
\end{proof}

\end{document}